\pdfoutput=1

\documentclass[reqno,a4paper,final]{amsart}

\usepackage[utf8]{inputenc}
\usepackage[T1]{fontenc}
\usepackage[english]{babel}
\usepackage{ifthen}
\usepackage{enumitem}
\usepackage{dsfont}
\usepackage{mathtools}
\usepackage{amssymb}
\usepackage[babel]{csquotes}
\usepackage[backend=biber,style=alphabetic]{biblatex}
\usepackage[protrusion=true,expansion=true,babel=true,final]{microtype}
\usepackage[unicode,bookmarksopen]{hyperref}

\numberwithin{equation}{section} 

\newtheorem{lemma}{Lemma}[section]
\newtheorem{corollary}[lemma]{Corollary}
\newtheorem{proposition}[lemma]{Proposition}
\newtheorem{theorem}[lemma]{Theorem}

\theoremstyle{definition}
\newtheorem{definition}[lemma]{Definition}
\newtheorem{example}[lemma]{Example}
\newtheorem{remark}[lemma]{Remark}

\newlist{thm_enum}{enumerate}{1}
\setlist[thm_enum]{label=\normalfont(\alph*)}
\newlist{def_enum}{enumerate}{1}
\setlist[def_enum]{label=\normalfont(\roman*)}
\newlist{equiv_enum}{enumerate}{1}
\setlist[equiv_enum]{label=\normalfont(\roman*)}

\newcommand{\IN}{\mathbb{N}}
\newcommand{\IR}{\mathbb{R}}
\newcommand{\IC}{\mathbb{C}}
\newcommand{\IE}{\mathbb{E}}

\newcommand{\abs}[1]{\left\lvert#1\right\rvert}
\newcommand{\normalabs}[1]{\lvert#1\rvert}

\newcommand{\norm}[1]{\left\lVert#1\right\rVert}
\newcommand{\normalnorm}[1]{\lVert#1\rVert}
\newcommand{\biggnorm}[1]{\biggl\lVert#1\biggr\rVert}

\newcommand{\R}[2][\empty]{
	\ifthenelse{\equal{#1}{\empty}}
		{\mathcal{R}\left\{#2\right\}}
		{\mathcal{R}_{#1}\left\{#2\right\}}
}

\makeatletter
\newcommand{\LeftEqNo}{\let\veqno\@@leqno}
\makeatother

\renewcommand{\d}{\mathop{}\!d}
\renewcommand{\Re}{\operatorname{Re}}

\renewcommand{\epsilon}{\varepsilon}

\let\temp\phi
\let\phi\varphi
\let\varphi\temp

\DeclareMathOperator{\E}{\IE}

\DeclareMathOperator{\Id}{Id}

\DeclareMathOperator{\supp}{supp}

\DeclareMathOperator{\Tr}{Tr}
\DeclareMathOperator{\dist}{dist}

\DeclareMathOperator{\MaxReg}{MR}
\DeclareMathOperator{\BV}{BV}
\DeclareMathOperator{\Ell}{Ell}

\allowdisplaybreaks[4]

\DeclareSourcemap{
  \maps[datatype=bibtex, overwrite]{
    \map{
      \step[fieldset=abstract, null]
      \step[fieldsource=entrykey,match=Mer99,final] %
      \step[fieldset=shorthand,fieldvalue=LeM99] %
    }
  }
}

\renewbibmacro{publisher+location+date}{%
	\printlist{publisher}
	\setunit*{\addcomma\space}
	\printlist{location}
  	\setunit*{\addcomma\space}
  	\usebibmacro{date}
\newunit} 

\newbibmacro{string+doi+url}[1]{%
	\iffieldundef{doi}{%
			\iffieldundef{url}{#1}{\href{\thefield{url}}{#1}}%
		}%
		{\href{http://dx.doi.org/\thefield{doi}}{#1}}
}%

\renewbibmacro{in:}{}

\DeclareFieldFormat*{title}{\usebibmacro{string+doi+url}{\mkbibemph{#1}}}
\DeclareFieldFormat*{booktitle}{#1}
\DeclareFieldFormat[article]{volume}{\textbf{#1}\addcomma}
\DeclareFieldFormat[article]{number}{\addspace no.~#1}
\DeclareFieldFormat[article]{pages}{#1}
\DeclareFieldFormat{journaltitle}{#1} 
\DeclareFieldFormat{url}{}
\DeclareFieldFormat{eprint}{Preprint on arXiv: \href{http://arxiv.org/abs/#1}{#1}} 

\renewcommand*{\bibfont}{\footnotesize}

\begin{document}

\title[Maximal Regularity for Elliptic Operators of Bounded Variation]{Non-Autonomous Maximal Regularity for Forms Given by Elliptic Operators of Bounded Variation}

\begin{abstract}
	We show maximal $L^p$-regularity for non-autonomous Cauchy problems provided the trace spaces are stable in some parameterized sense and the time dependence is of bounded variation. In particular, on $L^2$, we obtain for all $p \in (1,2]$ maximal $L^p$-regularity for non-autonomous elliptic operators with measurable coefficients.    
\end{abstract}

\author{Stephan Fackler}
\address{Institute of Applied Analysis, Ulm University, Helmholtzstr.\ 18, 89069 Ulm}
\email{stephan.fackler@uni-ulm.de}
\thanks{This work was supported by the DFG grant AR 134/4-1 ``Regularität evolutionärer Probleme mittels Harmonischer Analyse und Operatortheorie''.}

\maketitle

\section{Introduction}
	
	Let $X$ be a Banach space and $(A(t))_{t \in [0,T)}$ closed operators on $X$. For an inhomogenity $f\colon [0,T) \to X$ and an initial value $u_0 \in X$ we consider the non-autonomous Cauchy problem
	\begin{equation*}
		\LeftEqNo
		\label{eq:nacp}\tag{NACP}
		\left\{
		\begin{aligned}
			\dot{u}(t) + A(t)u(t) & = f(t) \\
			u(0) & = u_0.
		\end{aligned}
		\right.
	\end{equation*}
	For $p \in (1, \infty)$ we say that~\eqref{eq:nacp} has \emph{maximal $L^p$-regularity} for $u_0 \in X$ if for all $f \in L^p([0,T);X)$ there exists a unique solution in $\MaxReg_{p}^{A}([0,T))$, the space of all measurable functions $u\colon [0,T) \to X$ with $\dot{u} \in L^p([0,T);X)$, $u(t) \in D(A(t))$ for almost all $t \in [0,T)$ and $A(\cdot)u(\cdot) \in L^p([0,T);X)$. Of particular interest for applications is the case where $A(t)$ realize elliptic operators. Maximal regularity can then be used in various ways to show existence of quasilinear equations and to study the asymptotic behaviour of their solutions~\cite{Pru02}.
	
	In the autonomous case $A(t) = A$ maximal regularity is well understood. A closed operator $A\colon D(A) \to X$ has maximal $L^p$-regularity for $u_0 \in X$ if and only if $A$ is $\mathcal{R}$-sectorial and $u_0$ lies in the real interpolation space $\Tr_p A \coloneqq (D(A), X)_{1/p,p}$. The notation $\Tr_p A$ is justified by the fact that in the autonomous case one has $\MaxReg_p^A([0,T]) \hookrightarrow C([0,T];\Tr_p A)$ and, further, that for every $x \in \Tr_p A$ one can find $u \in \MaxReg_p^A([0,T])$ with $u(0) = x$~\cite[Theorem~III.4.10.2]{Ama95}.
	
	The easiest setting is that of non-autonomous forms on Hilbert spaces. Here we consider a Hilbert space $V$ densely embedded into a second Hilbert space $H$. This induces the Gelfand triple $V \hookrightarrow H \hookrightarrow V'$. Further, one has non-autonomous bounded coercive sesquilinear forms $a\colon [0,T] \times V \times V \to \IC$, i.e.\ $a(t,\cdot,\cdot)$ is sesquilinear for all $t \in [0,T]$ and satisfies for some $\alpha, M > 0$ and all $u, v \in V$
	\begin{equation*}
		\tag{A}
		\label{eq:form_assumptions}
		\begin{split}
			\abs{a(t,u,v)} & \le M \norm{u}_V \norm{v}_V, \\
			\Re a(t,u,u) & \ge \alpha \norm{u}_V^2.
		\end{split}
	\end{equation*}
	For fixed $t \in [0,T]$ the sesquilinear form $a(t,\cdot,\cdot)$ induces a bounded operator $\mathcal{A}(t)\colon V \to V'$, which is also an unbounded operator on $V'$. We denote its part in $H$ by $A(t)$. By a classical result of Lions~\cite[p.~513, Theorem~2]{DauLio92}, the problem~\eqref{eq:nacp} for $\mathcal{A}$ satisfies maximal $L^2$-regularity if $t \mapsto a(t,u,v)$ is measurable for all $u, v \in V$. However, maximal $L^2$- or $L^p$-regularity for the operator $A(t)$ on $H$ is far more involved. For some time there was the hope that maximal $L^2$-regularity for $A$ holds for all $u_0 \in \Tr_2 A(0)$ and measurable forms. Requiring additionally the symmetry of $a$, i.e.\ $a(t,u,v) = \overline{a(t,v,u)}$ for all $t \in [0,T]$ and $u, v \in V$, this problem was explicitly asked by Lions for $u_0 = 0$~\cite[p.~68]{Lio61}. Dier observed that in absence of the Kato square root property maximal $L^2$-regularity can fail due to a single jump of the form~\cite[Section~5.2]{Die14}. Later, the author gave Hölder continuous counterexamples to Lions' original question~\cite{Fac16c}. From the positive results in~\cite{DieZac16} and the counterexamples in~\cite{Fac16c} it is now understood that the critical case for the time regularity of $t \mapsto a(t, \cdot, \cdot)$ is $H^{1/2}(\mathcal{L}(V,V'))$ in the Sobolev scale. In fact, positive results for higher regularities can be found in~\cite{DieZac16}, whereas the forms in~\cite{Fac16c} give counterexamples for lower regularities. What remains open as of now is the critical case of $H^{1/2}$-regularity.
	
	Note that in this case the Sobolev index of $\mathcal{A}(\cdot)$ is equal to zero. An easier example of the same index is however partially understood. Namely, maximal $L^2$-regularity was shown by Dier assuming $W^{1,1}$-regularity and symmetry of the forms~\cite{Die15} (for an alternative proof see~\cite{MenHaf16}). In fact, the result even holds for forms of bounded variation. This is particularly interesting as such forms may have jumps. In view of the counterexamples maximal $L^2$-regularity can fail for general forms in this setting. However, in terms of applications this result is not satisfactory as it excludes elliptic operators in divergence form with non-symmetric coefficients.
	
	As a corollary of our main result we will see that maximal $L^p$-regularity in the optimal range $p \in (1,2]$ in fact holds for \emph{non-symmetric} elliptic operators of bounded variation and for a very broad range of domains and boundary conditions. Note that this result faces some critical points of the theory: it deals at the same time with the crucial example of elliptic operators, regularity with critical Sobolev index and infinitely many jumps. Moreover, we are able to deal with the case $p \neq 2$.
	
	We now proceed as follows: in the next section we present the general strategy of our proof towards the existence of a solution. We then deal with the uniqueness of solutions and the main results in the following sections, whereas the proofs of some more technical estimates are postponed to later sections. In the last section we discuss the optimality of our results.
	
\section{Existence -- Strategy of Proof}
	
	In this section we introduce the main ideas and concepts relevant for the proof of our maximal regularity result. Observe that the case of bounded variation in particular includes piecewise constant operator functions. We use the following handy definition.
	
	\begin{definition}
		A non-autonomous operator $A(\cdot)$ is called \emph{a step operator} if there exists a partition $0 = t_0 < t_1 < \cdots < t_N = T$ and closed operators $A_1, \ldots, A_N$ such that for almost every $t \in [0,T)$
			\begin{equation*}
				A(t) = \sum_{k=1}^N A_k \mathds{1}_{[t_{k-1}, t_k)}(t).
			\end{equation*}
		Analogously, we define \emph{step forms} as piecewise constant forms with common domain.
	\end{definition}

	To simplify matters, we introduce our main ideas only in the form setting. However, our results will hold for more general non-autonomous operators. The case of step operators can be solved by iterating the autonomous result. In particular, for $L^p$-maximal regularity to hold one needs by the autonomous case the inclusion $\Tr_p(A_k) \subset \Tr_p(A_{k+1})$. Since we may interchange the order of two steps, for a reasonable result all trace spaces must agree. Further, we have $\Tr_2(A_k) = [D(A_k),H]_{\frac{1}{2}}$ in the Hilbert space case. If $A$ is induced by a symmetric form, we have $[D(A),H]_{\frac{1}{2}} = V$~\cite[Example~p.~45]{Are04}. Hence, we need $[D(A_k),H]_{\frac{1}{2}} = V$ for all $k \in \IN$, i.e.\ the operators $A_k$ satisfy the so-called \emph{Kato square root property}. It is known that the Kato square root property may fail for general forms~\cite{McI72}. However, by the celebrated solution of Kato's problem first proven in the full space case in~\cite{AHL+02} it holds true for elliptic operators in divergence form under a broad class of domains and boundary conditions.
			
	Now, more explicitly, recall that in the autonomous case $A(t) = A$ the unique mild solution of~\eqref{eq:nacp} is given by the variation of constants formula
	\begin{equation*}
		u(t) = \int_0^t e^{-(t-s) A} f(s) \d s + e^{-t A} u_0.
	\end{equation*}
	Hence, the difference of two solutions $u_1$ and $u_2$ for different operators $A_1$ and $A_2$, but the same initial value $u_0 \in H$ and inhomogenity $f \in L^1([0,T];H)$, is given by
	\begin{equation}
		\label{eq:difference}
		u_2(t) - u_1(t) = \int_0^t (e^{-(t-s) A_2} - e^{-(t-s) A_1}) f(s) \d s + (e^{-t A_2} - e^{-t A_1}) u_0.
	\end{equation}
	Let us for the moment suppose that for this difference an estimate of the form 
		\begin{equation*}
			\tag{E}
			\label{eq:difference_estimate}
			\norm{\dot{u}_2 - \dot{u}_1}_{L^2([0,T];H)} \lesssim \norm{\mathcal{A}_2 - \mathcal{A}_1} (\norm{f}_{L^2([0,T];H)} + \norm{u_0}_V)
		\end{equation*} 
	holds, where the implicit constant is independent of the operators. As a first step we obtain a good a priori estimate for step forms in terms of their variation. 
	
	\begin{definition}
		Let $E$ be a Banach space and $f\colon [0,T] \to E$. The \emph{variation} of $f$ is
			\begin{equation*}
				\norm{f}_{\BV([0,T]; E)} \coloneqq \sup_{0 = t_0 < t_1 < \cdots < t_N = T} \sum_{k=1}^N \norm{f(t_k) - f(t_{k-1})}_{E},
			\end{equation*}
		where the supremum is over all partitions $0 = t_0 < t_1 < \cdots < t_N = T$ of $[0,T]$.
	\end{definition}
	
	Let $B(t) = \sum_{k=1}^N A_k \mathds{1}_{[t_{k-1}, t_k)}(t)$ be induced by a step form. Replacing $B(t)$ by 
	\begin{equation*}
		B_1(t) = \sum_{k=1}^{N-2} A_k \mathds{1}_{[t_{k-1}, t_k)}(t) + A_{N-1} \mathds{1}_{[t_{N-2}, t_N)}(t),
	\end{equation*}
	i.e.\ loosing the last jump and staying constant there instead, the difference of the solutions $u = u^0$ and $u^1$ of the respective equations~\eqref{eq:nacp} satisfies by~\eqref{eq:difference_estimate}
	\begin{align*}
		\normalnorm{\dot{u} - \dot{u}^1}_{L^2([0,T);H)} & = \normalnorm{\dot{u} - \dot{u}^1}_{L^2([t_{N-1}, T);H)} \\
		& \lesssim \norm{\mathcal{A}_N - \mathcal{A}_{N-1}} ( \norm{f}_{L_2([t_{N-1}, T);H)} + \norm{u(t_{N-1})}_V ).
	\end{align*}
	We now iterate the previous argument. By subsequently loosing the last jump of the previous operator, we obtain the operator functions
	\begin{equation*}
	 	B_m(t) = \sum_{k = 1}^{N - m - 1} A_k \mathds{1}_{[t_{k-1}, t_k)}(t) + A_{N - m} \mathds{1}_{[t_{N - m - 1}, t_{N})}(t).
	\end{equation*} 
	Let $u^m$ be the corresponding solution of~\eqref{eq:nacp}. Iterating the estimate gives
	\begin{equation}
		\label{eq:derivative}
		\begin{split}
    		\MoveEqLeft \normalnorm{\dot{u}}_{L^2([0,T);H)} \le \sum_{k=1}^{N-1} \normalnorm{\dot{u}^k - \dot{u}^{k-1}}_{L^2([t_{N-k},T);H)} + \normalnorm{\dot{u}^{N-1}}_{L^2([0,T];H)} \\
    		& \lesssim \sum_{k=1}^{N-1} \norm{\mathcal{A}_{N-k+1} - \mathcal{A}_{N-k}} (\norm{f}_{L^2([t_{N-k},T);H)} + \norm{u(t_{N-k})}_V) \\
			& \qquad + \normalnorm{\dot{u}^{N-1}}_{L^2([0,T];H)} \\
    		& \le \norm{\mathcal{A}}_{\BV([0,T];\mathcal{L}(V,V'))} (\norm{f}_{L^2([0,T);H)} + \norm{u}_{L^{\infty}([0,T];V)}) \\
			& \qquad + \normalnorm{\dot{u}^{N-1}}_{L^2([0,T];H)}.
		\end{split}
	\end{equation}
	This shows that we obtain an a priori estimate for the derivatives provided we have an estimate for the solution in $L^{\infty}([0,T];V)$. In fact, we will establish such an estimate under a parameterized variant of the Kato square root property. Moreover, exactly this parameterization can be generalized to a Banach space setting. In the non-Hilbert space setting we work with $\mathcal{R}$-boundedness on UMD spaces. For these concepts we refer to~\cite{KunWei04} and~\cite{DHP03}.
	
	\begin{definition}\label{def:parameterized_kato}
		Let $X$ be a UMD space and $p \in (1, \infty)$. An \emph{$L^p$-trace parameterization} in $X$ is the datum $(\mathcal{P}, F_1 , F_2, E, U, O)$ for complex Banach spaces $E$, $F_1$ reflexive and $F_2$ with embeddings $F_1 \hookrightarrow X \hookrightarrow F_2$, $U \subset O \subset E$ subsets with $U$ closed and convex and a bounded linear operator $\mathcal{P}\colon E \to \mathcal{L}(F_1,F_2)$ satisfying
		\begin{def_enum}
			\item\label{trace_param:distance} $\dist(U, O^c) > 0$.
			\item\label{trace_param:r_boundedness} There exist $\phi \in (0,\frac{\pi}{2})$ such that for all $x \in O$ the operators $\mathcal{P}(x)$ seen as unbounded operators on $F_2$ and the parts $\mathcal{P}(x)_{|X}$ of $\mathcal{P}(x)$ in $X$, i.e.\ $D(\mathcal{P}(x)_{|X}) = \{ z \in F_1: \mathcal{P}(x)z \in X \}$, satisfy the spectral inclusions
				\begin{equation*}
					\sigma(\mathcal{P}(x)) \cup \sigma(\mathcal{P}(x)_{|X}) \subset \{ z \in \IC: \abs{\arg z} < \phi \}
				\end{equation*}
				together with the uniform resolvent bound 
				\begin{equation*}
					\sup_{x \in O} \sup_{\abs{\arg z} \ge \phi} \norm{(1+\abs{z})R(z,\mathcal{P}(x))} < \infty
				\end{equation*}
				and the $\mathcal{R}$-bound
				\begin{equation*}
					\sup_{x \in O} \R{(1 + \abs{z})R(z, \mathcal{P}(x)_{|X}): \abs{\arg z} \ge \phi} < \infty.
				\end{equation*}
			\item\label{trace_param:trace} The trace spaces $\Tr_p \mathcal{P}(x)_{|X} = (D(\mathcal{P}(x)_{|X}),X)_{1/p,p}$ are independent of $x \in O$ and their norms are uniformly comparable.
		\end{def_enum}
		A family $(A_i)_{i \in I}$ of closed operators on $X$ is \emph{$L^p$-parameterized} by a parameterization $(\mathcal{P}, F_1 , F_2, E, U, O)$ if for all $i \in I$ one has $A_i =  \mathcal{P}(x)_{|X}$ for some $x \in U$.
	\end{definition}
	
	\begin{remark}\label{rem:uniform_estimate}
		Since we assumed in Definition~\ref{def:parameterized_kato} that $X$ is a UMD space, property~\ref{trace_param:r_boundedness} implies that $\mathcal{P}(x)_{|X}$ has maximal $L^p$-regularity for all $p \in (1, \infty)$. Further, it follows from the explicit dependence in the vector-valued Mihlin multiplier theorem~\cite[Corollary~4.4]{GirWei03b} that there exists $C > 0$ such that for all $x \in O$, $T \in (0,\infty]$, $f \in L^p([0,T);X)$ and $u_0 \in \Tr_p \mathcal{P}(x)_{|X}$ the solution $u$ of~\eqref{eq:nacp} satisfies
		\begin{equation*}
			\norm{u}_{W^{1,p}([0,T);X)} + \normalnorm{\mathcal{P}(x)_{|X} u}_{L^p([0,T);X)} \le C (\norm{f}_{L^p([0,T);X} + \norm{u_0}_{\Tr \mathcal{P}(x)_{|X}}).
		\end{equation*}
		Further, by the trace method for real interpolation we have
		\begin{equation}
			\label{eq:trace_embedding}
			\norm{u}_{C([0,T);\Tr \mathcal{P}(x)_{|X})} \lesssim \norm{f}_{L^p([0,T);X)} + \norm{u_0}_{\Tr \mathcal{P}(x)_{|X}}.
		\end{equation}
	\end{remark}
	
	 A fundamental example is given by elliptic operators in divergence form.
	 
	 \begin{example}\label{ex:trace_elliptic}
	 	For $M, \epsilon > 0$ we denote by $\Ell(\epsilon, M)$ the set of all elliptic divergence form operators on $L^2(\IR^n)$ induced by
		\begin{equation*}
			(u,v) \mapsto \sum_{i,j=1}^n \int_{\IR^n} a_{ij} \partial_i u \overline{\partial_j v}
		\end{equation*}
		for some $(a_{ij}) \in C(\epsilon, M)$, the set of all complex coefficients $a_{ij} \colon \IR^n \to \IC^{n \times n}$ with $\Re \sum_{i,j=1}^n a_{ij} \xi_i \overline{\xi}_j \ge \epsilon \abs{\xi}^2$ and $\norm{a_{ij}}_{\infty} \le M$. Further, let $E = L^{\infty}(\IR^n; \IC^{n \times n})$ and
		\begin{align*}
			\mathcal{P}\colon L^{\infty}(\IR^n;\IC^{n \times n}) & \to \mathcal{L}(H^1(\IR^n),H^{-1}(\IR^n)) \\
			(a_{ij}) & \mapsto \biggl[ u \mapsto \bigl[ v \mapsto \sum_{i,j=1}^n \int_{\IR^n} a_{ij} \partial_i u \overline{\partial_j v} \bigr] \biggr].
		\end{align*}
		Then $\Ell(\epsilon, M)$ is $L^2$-parameterized by 
		\begin{equation*}
			(\mathcal{P}, H^1(\IR^n), H^{-1}(\IR^n), L^{\infty}(\IR^n; \IC^{n \times n}), C(\epsilon, M), C(\epsilon/2, 2M)).
		\end{equation*} 
		Note that property~\ref{trace_param:distance} is clear, whereas~\ref{trace_param:r_boundedness} is always satisfied for uniformly bounded and coercive forms: $\mathcal{R}$-boundedness reduces to mere boundedness on Hilbert spaces and the estimates for sectorial operators only depend on the constants $\epsilon$ and $M$ in the definition of forms. The crucial property~\ref{trace_param:trace} is a consequence of the solution of the Kato square root problem in~\cite[Theorem~6.1]{AHL+02}. Here we use that $[D(A),H]_{1/2,2} = [D(A),H]_{1/2}$ isometrically, which in turn is uniformly equivalent to $D(A^{1/2}) = H^1(\IR^n)$ because of~\cite[Proposition~2.5]{Fac15c} and the fact that all operators induced by forms satisfy $\normalnorm{A^{it}} \le e^{\frac{\pi}{2} \abs{t}}$~\cite[Theorem~4.29]{Lun09}.
	 \end{example}
	 
	 For general sesquilinear forms one has the following slightly weaker positive result.
	 
	 \begin{example}\label{ex:form}
	 	Let $V,H$ be complex Hilbert space with dense embeddings $V \hookrightarrow H$ and let $\mathcal{F}(\epsilon, M)$ for $\epsilon, M > 0$ be the family of all operators in $\mathcal{L}(V,V')$ induced by some sesquilinear form $a\colon V \times V \to \IC$ with $\Re a(u,u) \ge \epsilon \norm{u}_V^2$ and $\abs{a(u,v)} \le M \norm{u}_V \norm{v}_V$ for all $u, v \in V$. Consider the trivial parameterization $\mathcal{P}\colon \mathcal{L}(V,V') \to \mathcal{L}(V,V')$ given by the identity mapping. We show that for $p \in (1,2)$
		\begin{equation*}
			(\mathcal{P},V,V',\mathcal{L}(V,V'), \mathcal{F}(\epsilon,M), \mathcal{F}(\epsilon/2, M + \epsilon/2))
		\end{equation*} 
		is an $L^p$-parameterization of all operators obtained as part in $H$ of some element in $\mathcal{F}(\epsilon,M)$. First, \ref{trace_param:distance} follows from the inclusion $\mathcal{F}(\epsilon,M) + B(0,\epsilon/2) \subset \mathcal{F}(\epsilon/2, M + \epsilon/2)$. Secondly, \ref{trace_param:r_boundedness} holds universally for forms as discussed in Example~\ref{ex:trace_elliptic}. Let us now verify~\ref{trace_param:trace}. Suppose that $A$ is induced by some element in $\mathcal{F}(\epsilon,M)$. Then the same holds for its Hilbert space adjoint $A^*$ and for the real part $\Re \frac{1}{2} (A+A^*)$. By a result of Kato one has $D(A^{\alpha}) = D((\Re A)^{\alpha})$ for all $\alpha \in (0, \frac{1}{2})$~\cite[Theorem~3.1]{Kat61b}. Further, their norms are comparable with constants only depending on $\alpha$, $M$ and $\epsilon$. Since the Kato square root property holds isometrically for self-adjoint operators such as $\Re A$ and since the fractional domains are uniformly comparable to complex interpolation spaces (see Example~\ref{ex:form}), for $\alpha \in (0,\frac{1}{2})$ the reiteration theorem for the complex method gives the uniform equivalence~\cite[Theorem~4.6.1]{BerLoe76}
		\begin{align*}
			[H,D(A)]_{\alpha} & = D(A^{\alpha}) = D((\Re A)^{\alpha}) = [H, D(\Re A)]_{\alpha} \\
			& = [H,[H,D(\Re A)]_{\frac{1}{2}}]_{2\alpha} = [H,V]_{2\alpha}.
		\end{align*}
		By the reiteration theorem for the real method~\cite[1.10.3, Theorem~2]{Tri78} we have
		\begin{align*}
			(H,D(A))_{1-\frac{1}{p},p} & = (H,[H,D(A)]_{1-\frac{1}{2}(\frac{1}{2} + \frac{1}{p})})_{\frac{4p-4}{3p-2},p} = (H,[H,V]_{\frac{3p-2}{2p}})_{\frac{4p-4}{3p-2},p} \\
			& = (H, V)_{2-\frac{2}{p},p}.
		\end{align*}
		The first equivalence is uniform as a consequence of the constants appearing in the proof of the reiteration theorem and the estimate obtained at the end of the proof of~\cite[1.10.3, Theorem~1]{Tri78}.
	 \end{example}
	 
	 Certain divergence form operators on $L^q$-spaces seem to fit in our framework.
	 
	 \begin{remark}
	 	For $q \in (1, \infty)$ and coefficients $A \in C(\epsilon,M)$ we define the operator $\mathcal{B}_q(A)\colon W^{1,q}(\IR^n) \to W^{-1,q}(\IR^n)$ as
		\begin{equation*}
			\langle \mathcal{B}_q(A) u, v \rangle_{W^{-1,q},W^{1,q}} \coloneqq \int_{\IR^n} A \nabla u \overline{\nabla v}.
		\end{equation*}
		For $q$ in an interval $I_A$ containing $2$ and depending only on the ellipticity constants and the dimension $n$ one can show that $\mathcal{B}_q(A)$ is sectorial. This follows from the results in~\cite[Section~4]{Aus07} and~\cite[Section~6]{DisElsReh16}. The part of $\mathcal{B}_q(A)$ on $L^q(\IR^n)$ then is $\mathcal{R}$-sectorial by~\cite[Theorem~5.1]{Aus07} and~\cite[Theorem~5.3]{KalWei01}. Hence, if one defines $\mathcal{P}\colon L^{\infty}(\IR^n; \IC^{n \times n}) \to \mathcal{L}(W^{1,q}(\IR^n), W^{-1,q}(\IR^n))$ as $A \mapsto \mathcal{B}_q(A)$,
		\begin{equation*}
			(\mathcal{P},W^{1,q}(\IR^n),W^{-1,q}(\IR^n),L^{\infty}(\IR^n;\IC^{n \times n}), C(\epsilon, M), C(\epsilon/2, 2M))
		\end{equation*}
		is a candidate for an $L^p$-trace parameterization. However, several details need to be checked. First, the above results should only depend on $q$, $n$ and the ellipticity constants. This is not explicitly stated in the cited results. Secondly, for $q \in I_A$ the operator satisfies the Kato square root property $[L^q(\IR^n),D(B_q)]_{1/2} = D(B_q^{1/2}) = W^{1,q}(\IR^n)$. Hence, by using an analogous reiteration argument as in Example~\ref{ex:form}, the space $\Tr_p B_q$ is independent of $B_q$ for $p \in (1,2)$. However, we do not know whether the independence of $\Tr_2 A$ holds as well, as it is the case for $q = 2$. Note that if $q \neq 2$, then under general assumptions one has $[L^q,D(A_q)]_{1/2} \neq (L^q,D(A_q))_{1/2,2}$~\cite[Corollary~3.2]{KucWei05}.
		
		Note that if one assumes more spatial regularity on the coefficients, one may restrict to a smaller representation and the properties of Definition~\ref{ex:trace_elliptic} become more easy to verify.	 
	\end{remark}

	 The following a priori estimate in $L^{\infty}([0,T];\Tr_p A)$ is based on Section~\ref{sec:a_priori_bounded}. If $A(\cdot)$ is parameterized by some $L^p$-trace parameterization, then $\Tr_p A(t)$ is independent of $t$. Therefore we may simply write $\Tr_p A$.
	
	\begin{theorem}\label{thm:a_priori_bounded}
		Let $(\mathcal{P}, F_1 , F_2, E, U, O)$ be an $L^p$-trace parameterization. For every $R > 0$ there exists a constant $C = C(R)$ such that if $A(t) = \mathcal{P}(x(t))$ for $x\colon [0,T] \to E$ is a step function with $\norm{x}_{\BV} \le R$, then the solution $u$ of~\eqref{eq:nacp} satisfies
			\[
				\norm{u}_{C([0,T];\Tr_p A)} \le C (\norm{f}_{L^p([0,T];H)} + \norm{u_0}_{\Tr_p A}).
			\]
	\end{theorem}
	\begin{proof}
		Using the same notation as in~\eqref{eq:derivative}, Proposition~\ref{prop:linfity_initial} and Proposition~\ref{prop:linfty_inhomogenity} imply that for some $K > 0$ and all $N \in \IN$
		\begin{equation*}
    		\begin{split}
        		\MoveEqLeft \norm{u}_{C([0,T];\Tr_p A)} \le \sum_{k=1}^{N-1} \normalnorm{u^k - u^{k-1}}_{C([t_{N-k},T];\Tr_p)} + \normalnorm{u^{N-1}}_{C([0,T];V)} \\
        		& \le K \sum_{k=1}^{N-1} \norm{x_{N-k+1} - x_{N-k}} (\norm{f}_{L^p([t_{N-k},T];X)} + \norm{u(t_{N-k})}_{\Tr_p A}) \\
    			& \qquad + \normalnorm{u^{N-1}}_{C([0,T];\Tr_p A)} \\
        		& \le K \norm{x}_{\BV([0,T];E)} (\norm{f}_{L^p([0,T];X)} + \norm{u}_{C([0,T];\Tr_p A)}) \\
    			& \qquad + \normalnorm{u^{N-1}}_{C([0,T];\Tr_p A)}.
    		\end{split}
    	\end{equation*}
		Since $u^{N-1}$ solves the autonomous equation $\dot{u}(t) + A_1 u(t) = f(t)$, we have for $\norm{x}_{\BV} \le \frac{1}{2K}$ the estimate
		\begin{equation}
			\label{eq:linfty_estimate}
			\norm{u}_{C([0,T];\Tr_p)} \lesssim \norm{f}_{L^p([0,T];X)} + \norm{u_0}_{\Tr_p A}.
		\end{equation}
		Here the implicit constant does not depend on $T$. Now, let $x(t) = \sum_{k=1}^N \mathds{1}_{[t_{k-1},t_k)} x_k$ be an arbitrary step function with $\norm{x}_{\BV} \le R$. Choose $N_1 \in \IN$ maximal with $\sum_{k=2}^{N_1} \norm{x_{k} - x_{k-1}} \le \frac{1}{2K}$. Now, choose a natural number $N_2 > N_1$ maximal with $\sum_{k=N_1 + 2}^{N_2} \norm{x_{k} - x_{k-1}} \le \frac{1}{2K}$. We iterate this procedure finitely often until we have $N_M = N$ for some $M \in \IN$. Notice that $M$ is uniformly bounded by $2RK$. In fact, assume that $M > 2RK$. Then, setting $N_0 = 0$, we obtain the contradiction
		\begin{equation*}
			\norm{x}_{BV} = \sum_{l=1}^M \sum_{k=N_{l-1}+2}^{N_l + 1} \norm{x_k - x_{k-1}} > M \cdot \frac{1}{2K} > R.
		\end{equation*}
		The result now follows by iterating estimate~\eqref{eq:linfty_estimate} or for big jumps the analogous autonomous result~\eqref{eq:trace_embedding} uniformly bounded many times.
	\end{proof}
	
	Using this in~\eqref{eq:derivative}, we get for $\dot{u}$ the a priori estimate
	\begin{equation}
		\label{eq:a_priori_derivative}
		\begin{split}
			\norm{\dot{u}}_{L^2([0,T];H)} & \lesssim \norm{\mathcal{A}}_{\BV([0,T];\mathcal{L}(V,V'))} (\norm{f}_{L^2([0,T);H)} + C) \\
			& + \norm{f}_{L^2([0,T];H)} + \norm{u_0}_V.
		\end{split}
	\end{equation}
	Recall that this estimate is based on the validity of~\eqref{eq:difference_estimate}. It follows from Proposition~\ref{prop:l2} that~\eqref{eq:difference_estimate} indeed holds if $\mathcal{A}(t)$ is parameterized in the sense of Definition~\ref{def:parameterized_kato} and if one replaces $\norm{\mathcal{A}_2 - \mathcal{A}_1}$ by the norm $\norm{x_1 - x_2}$ of their parameterizations. By the same reasoning this gives~\eqref{eq:a_priori_derivative} with $\norm{\mathcal{A}}_{\BV}$ replaced by the $\BV$-seminorm of a representation $x\colon [0,T] \to E$. The general case now follows from approximation.
	
	\begin{theorem}\label{thm:existence}
		Let $(\mathcal{P}, F_1 , F_2, E, U, O)$ be an $L^p$-trace parameterization in a UMD space $X$ for some $p \in (1, \infty)$. If $\mathcal{A}(t) = \mathcal{P}(x(t))$ for some $x(t)\colon [0,T] \to E$ with $\norm{x}_{\BV([0,T];E)} < \infty$, then for all $u_0 \in \Tr_p A$ \eqref{eq:nacp} has a solution in $\MaxReg_p^A([0,T])$.
	\end{theorem}
	\begin{proof}
		Note that $x$ is bounded and measurable. For every $n \in \IN$ consider the piecewise constant approximations 
		\begin{equation*}
			x_n(t) = \sum_{k=0}^{n-1} n \int_{\frac{k}{n}}^{\frac{k+1}{n}} x(s) \d s \mathds{1}_{[\frac{k}{n}, \frac{k+1}{n})}(t).
		\end{equation*}
		Since $U$ is convex and closed, the approximations $x_n$ take values in $U$ as well. Further, we have for $k \in \{0, \ldots, n-2\}$ 
		\begin{align*}
			\sum_{k=0}^{n-1} \biggnorm{x_n\biggl(\frac{k+1}{n}\biggr) - x_n \biggl(\frac{k}{n} \biggr)} = n \int_0^{\frac{1}{n}} \sum_{k=0}^{n-1} \biggnorm{x \biggl(s+\frac{k+1}{n}\biggr) - x \biggl(s+\frac{k}{n}\biggr)} \d s.
		\end{align*}
		Consequently, $\norm{x_n}_{\BV} \le \norm{x}_{\BV}$. Further, by Lebesgue's differentiation theorem $x_n(t) \to x(t)$ almost everywhere on $[0,T)$. It follows from the parameterized variant of~\eqref{eq:a_priori_derivative} that the solutions $u_n$ of~\eqref{eq:nacp} for $\mathcal{A}_n(t) = \mathcal{P}(x_n(t))$ and initial value $u_0$ satisfy the uniform estimate
		\begin{equation*}
			\norm{\mathcal{A}_n(\cdot) u_n(\cdot)}_{L^p([0,T];X)} + \norm{\dot{u}_n}_{L^p([0,T];X)} \lesssim \norm{f}_{L^p([0,T];X)} + \norm{u_0}_{\Tr_p A}.
		\end{equation*}
		Hence, after passing to a subsequence we may assume that $\mathcal{A}_n(\cdot) u_n(\cdot)$ and $\dot{u}_n$ converge weakly in $L^p([0,T];X)$ and that $u_n$ converges weakly to some $u \in L^p([0,T];F_1)$. Then $u \in W^{1,p}([0,T];X)$ and $\dot{u}$ agrees with the weak limit of $(\dot{u}_n)_{n \in \IN}$. Now, testing~\eqref{eq:nacp} against $g \in L^{p'}([0,T];F_2')$ we have 
		\begin{equation}
			\label{eq:weak_solution}
			\int_0^T \langle \dot{u}_n(t), g(t) \rangle_{F_2, F_2'} \d t + \int_0^T \langle \mathcal{A}_n(t) u_n(t), g(t) \rangle_{F_2, F_2'} \d t = \int_0^T \langle f(t), g(t) \rangle \d t.
		\end{equation}
		We have $\mathcal{A}_n'(\cdot) g(\cdot) \to \mathcal{A}'(\cdot) g(\cdot)$ in $L^{p'}([0,T];F_2')$. Hence, using duality in~\eqref{eq:weak_solution} and passing to the limit, we have $\dot{u} + \mathcal{A}(t)u(t) = f(t)$ and $\mathcal{A}(\cdot)u(\cdot) \in L^p([0,T];X)$.
	\end{proof}
	
\section{Uniqueness}

	We now come to the uniqueness of solutions. We are not able to prove a general uniqueness result in the setting of Theorem~\ref{thm:existence}. Nevertheless, for many applications we can rely on existing results. We now introduce the necessary terminology. Let $X$ be a Banach space. For an element $x \in X$ we define its duality set as
	\begin{equation*}
		\mathcal{J}(x) \coloneqq \{ x' \in X': \norm{x'} = 1, \langle x', x \rangle = \norm{x} \}.
	\end{equation*}
	By the Hahn--Banach theorem this set is always non-empty. Recall that an operator $A\colon D(A) \to X$ is \emph{accretive} if for all $x \in D(A)$ there exists some $x' \in \mathcal{J}(x)$ with $\Re \langle x', Ax \rangle \ge 0$. If $X$ is a Hilbert space, this is equivalent to the well-known condition $(Ax|x) \ge 0$ for all $x \in D(A)$. The following result is obtained by repeating the proof of~\cite[Proposition~3.2]{ACFP07} word by word.

	\begin{theorem}\label{thm:uniqueness}
		Let $(A(t))_{t \in [0,T)}$ be accretive operators on some Banach space $X$. For $p \in (1, \infty)$ suppose that $u_1, u_2 \in \MaxReg_p^A([0,T])$ solve~\eqref{eq:nacp}. Then $u_1  = u_2$.
	\end{theorem}
	
\section{Main Results}
	
	From our general results we obtain several concrete results. Note that if $A$ is induced by a form, then $A$ is accretive. For general forms Theorem~\ref{thm:existence}, Theorem~\ref{thm:uniqueness} and Corollary~\ref{ex:form} therefore give the following.
	
	\begin{corollary}[Maximal Regularity for Forms of Bounded Variation]\label{cor:elliptic_operator}
		Let $V, H$ be complex Hilbert spaces with dense embeddings $V \hookrightarrow H$ and $a\colon [0,T] \times V \times V \to \IC$ a non-autonomous form satisfying~\eqref{eq:form_assumptions} and $\norm{\mathcal{A}(\cdot)}_{\BV([0,T];\mathcal{L}(V,V'))} < \infty$. Then~\eqref{eq:nacp} has maximal $L^p$-regularity for all $p \in (1, 2)$ and all $u_0 \in (H,V)_{2-2/p,p}$.
	\end{corollary}
	
	Note that by Dier's counterexample the above result does not extend to exponents $p \ge 2$. However, Example~\ref{ex:trace_elliptic} gives maximal $L^2$-regularity for elliptic operators.
	
	\begin{corollary}[Elliptic Operators -- Full space case]
		Let $a_{ij}\colon [0,T] \times \IR^n \to \IC$ for $i, j = 1, \ldots, n$ be measurable coefficients such that for $M, \epsilon > 0$ the following holds.
		\begin{thm_enum}
			\item $\abs{a_{ij}(t,x)} \le M$ for almost every $(t,x) \in [0,T] \times \IR^n$ and all $i,j = 1, \ldots, n$.
			\item $\Re \sum_{i,j=1}^n a_{ij}(t,x) \xi_i \xi_j \ge \epsilon \abs{\xi}^2$ for almost every $(t,x) \in [0,T] \times \IR^n$ and all $\xi \in \IC^n$.
			\item $t \mapsto a_{ij}(t,\cdot) \in \BV([0,T];L^{\infty}(\IR^n))$ for all $i,j = 1, \ldots, n$.
		\end{thm_enum}
		Then one has maximal $L^2$-regularity for the non-autonomous form $a\colon [0,T] \times H^1(\IR^n) \times H^1(\IR^n) \to \IC$ given by
		\begin{equation*}
			a(t,u,v) = \sum_{i,j=1}^n \int_{\IR^n} a_{ij}(t,x) \partial_i u(x) \overline{\partial_j v(x)} \d x.
		\end{equation*} 
	\end{corollary}
	
	For mixed boundary conditions we imitate Example~\ref{ex:trace_elliptic} and use~\cite[Theorem~1]{AKM06}. Note that one can add lower order terms as expected.
	
	\begin{corollary}[Mixed Boundary Conditions]
		Let $\Omega \subset \IR^n$ be a smooth domain with either $\Omega$ or $\Omega^c$ bounded. Further let $\Sigma \subset \partial \Omega$ be open and $a_{ij}\colon [0,T] \times \IR^n \to \IC$ for $i, j = 0, \ldots, n$ be measurable such that for $M, \epsilon > 0$ the following holds.
		\begin{thm_enum}
			\item $\abs{a_{ij}(t,x)} \le M$ for almost every $(t,x) \in [0,T] \times \IR^n$ and all $i,j = 0, \ldots, n$.
			\item $\Re \sum_{i,j=1}^n a_{ij}(t,x) \xi_i \xi_j \ge \epsilon \abs{\xi}^2$ for almost every $(t,x) \in [0,T] \times \IR^n$ and all $\xi \in \IC^n$.
			\item $t \mapsto a_{ij}(t,\cdot) \in \BV([0,T];L^{\infty}(\IR^n))$ for all $i,j = 0, \ldots, n$.
		\end{thm_enum}
		With all traces interpreted in the Sobolev sense consider the form domain
		\begin{equation*}
			V = \{ u \in H^1(\Omega): \supp u_{|\partial \Omega} \subset \overline{\Sigma} \}.
		\end{equation*}
		Under the above assumptions we have maximal $L^2$-regularity for the non-autonomous form $a\colon [0,T] \times V \times V \to \IC$ given by
		\begin{align*}
			a(t,u,v) & = \sum_{i,j=1}^n \int_{\IR^n} a_{ij}(t,x) \partial_i u(x) \overline{\partial_j v(x)} + a_{j0}(t,x) u(x) \overline{\partial_j v(x)} \\
			& + a_{0k}(t,x) \partial_k u(x) \overline{v(x)} + a_{00}(t,x) u(x) \overline{v(x)} \d x.
		\end{align*}
	\end{corollary}
	
	More generally, the above result holds for bi-Lipschitz images of the above geometric configuration. The Kato square root property for mixed boundary conditions is in fact known for more general domains~\cite{EgeHalTol14}. However, these results do not state the dependence on the constants explicitly. Further, \cite[Theorem~1.3]{APM01} gives an analogous result for higher order elliptic systems.
	
	\begin{corollary}[Higher Order Elliptic Systems]
		Let $N \in \IN$ and $m \ge 2$. Let $a_{\alpha \beta} = (a_{\alpha \beta}^{ij})_{1 \le i,j \le N} \colon [0,T] \times \IR^n \to \IC^N \times \IC^N$ for multi-indices $\alpha, \beta \in \IN_0^n$ with $\abs{\alpha} = \abs{\beta} = m$ be measurable coefficients such that for $M, \lambda > 0$ the following holds.
		\begin{thm_enum}
			\item $\normalabs{a_{\alpha \beta}^{ij}(t,x)} \le M$ for almost every $(t,x) \in [0,T] \times \IR^n$ and all $i,j = 1, \ldots, n$ and $\abs{\alpha} = \abs{\beta} = m$.
			\item 
			For the self-adjoint part $b_{\alpha \beta} = \frac{1}{2} (a_{\alpha \beta} + a_{\beta \alpha}^*)$ one has
			\begin{equation*}
				\Re \sum_{\substack{\abs{\alpha} = \abs{\beta} = m \\ 1 \le i, j \le N}} b_{\alpha \beta}^{ij}(t,x) \xi_{\beta,i} \overline{\xi_{\alpha, i}} \ge \lambda \sum_{\substack{\abs{\alpha} = m \\ 1 \le i \le N}} \abs{\xi_{\alpha,i}}^2
			\end{equation*}
			for almost every $(t,x) \in [0,T] \times \IR^n$ and all $\xi_{\alpha,i} \in \IC$.
			\item $t \mapsto a_{\alpha \beta}^{ij}(t,\cdot) \in \BV([0,T];L^{\infty}(\IR^n))$ for all $i,j = 1, \ldots, n$ and $\abs{\alpha} = \abs{\beta} = m$.
		\end{thm_enum}
		Then one has maximal $L^2$-regularity for the non-autonomous form $a\colon [0,T] \times H^1(\IR^n;\IC^N) \times H^1(\IR^n;\IC^N) \to \IC$ given by
		\begin{equation*}
			a(t,u,v) = \sum_{\abs{\alpha} = \abs{\beta} = m} \int_{\IR^n} a_{\alpha \beta}(t,x) \partial^{\beta} u(x) \overline{\partial^{\alpha} v(x)} \d x.
		\end{equation*}
	\end{corollary}
	
\section{A priori Boundedness of the Solutions in the Trace Space}\label{sec:a_priori_bounded}

	In this subsection we prove the propositions used in Theorem~\ref{thm:a_priori_bounded}. We do this by treating the two summands in~\eqref{eq:difference} separately. As all following estimates are based on the same key ideas, we only give a detailed account once and will concentrate on the differences in the later proofs.
	
	\begin{proposition}\label{prop:linfty_inhomogenity}
		Let $(\mathcal{P}, F_1 , F_2, E, U, O)$ be an $L^p$-trace parameterization in a UMD-space $X$. If $u_{\mathcal{A}}$ solves $\dot{u} + \mathcal{A}u = f$ for $u_0 = 0$, then uniformly in $T \in (0, \infty)$
			\begin{align*}
				\normalnorm{u_{\mathcal{P}(x)} - u_{\mathcal{P}(y)}}_{C([0,T];\Tr_p A)} \lesssim \norm{x-y}_E \norm{f}_{L^p([0,T];H)} \qquad \text{for all } x,y \in U.
			\end{align*} 
	\end{proposition}
	\begin{proof}
		Fix some $x \in U$ and let $\mathcal{P}(x) = \mathcal{A}$. By definition one has $B(x,r) \subset O$ for some universal $r > 0$. Choose $\Delta x \in B(0,r)$ and let $\Delta \mathcal{A} = \mathcal{P}(\Delta x)$. Let $z \in \IC$ with $\abs{z} \le 1$. Observe that $x + z \Delta x \in B(x,r) \subset O$ and $\mathcal{A} + z \Delta \mathcal{A} = \mathcal{P}(x + z \Delta x)$. Recall that, by definition, for all $x \in O$ the spectrum of $\mathcal{P}(x)$ is contained in the sector $\{ z \in \IC \setminus \{ 0 \}: \abs{\arg z} \le \phi \}$. By the resolvent identity we have for $\abs{z} \le 1$ and $w$ outside this sector
			\begin{equation*}
				R(w, \mathcal{A} + z \Delta \mathcal{A}) = R(w, \mathcal{A}) + z R(w, \mathcal{A}) \Delta \mathcal{A} R(w, \mathcal{A} + z \Delta \mathcal{A}).
			\end{equation*}
		Hence, for $\abs{z}$ sufficiently small we have
			\begin{align*}
				R(w, \mathcal{A} + z \Delta \mathcal{A}) = (\Id - zR(w,\mathcal{A}) \Delta \mathcal{A})^{-1} R(w, \mathcal{A}).
			\end{align*}
		It follows from the von Neumann series representation that $z \mapsto R(w, \mathcal{A} + z\Delta \mathcal{A}) \in \mathcal{L}(F_2,F_1)$ is analytic. Hence, applying the elementary holomorphic functional calculus to the explicit formula for $u_{\mathcal{A} + z \Delta \mathcal{A}}$, we get for $f \in C_c^{\infty}((0,T);F_2)$
			\begin{align*}
				\MoveEqLeft u_{\mathcal{A} + z \Delta \mathcal{A}}(t) = \int_0^t e^{-(\mathcal{A} + z \Delta \mathcal{A})(t-s)} f(s) \d s \\
				& = \int_0^t \int_{\Gamma} e^{-(t-s) w} R(w, \mathcal{A} + z \Delta \mathcal{A}) f(s) \d w \d s. %
			\end{align*}
		It follows from this representation that $z \mapsto u_{\mathcal{A} + z\Delta \mathcal{A}} \in F_1$ is analytic. Hence, the coefficients $a_k$ of its series expansion are given by
			\begin{align*}
				a_k = \frac{1}{2\pi i} \int_{\abs{z} = r} u_{\mathcal{A} + z \Delta \mathcal{A}}(t) \frac{\d z}{z^{k+1}}
			\end{align*}
		for $r \in (0,1)$. We now estimate the coefficients in the stronger norm of $\Tr_p A$. The trace method for real interpolation~\cite[Section~1.2.2]{Lun95} and the uniform maximal regularity estimate give for $\mathcal{B} \in \mathcal{P}(O)$
		\begin{align*}
			\norm{u_{\mathcal{B}}(t)}_{\Tr_p A}^p & \lesssim \int_0^T \norm{u_{\mathcal{B}}(s)}_X^p + \norm{Bu_{\mathcal{B}}(s)}_X^p \d s \lesssim \int_0^T \norm{f(s)}_X^p \d s.
		\end{align*}
		As a consequence we have $\norm{a_k}_{\Tr_p A} \lesssim r^{-k} \norm{f}_{L^p([0,T];X)}$. It follows that the mapping $G\colon O \to C([0,T];\Tr_p A)$ given by $x \mapsto u_{\mathcal{P}(x)}$ is analytic and bounded on $U$~\cite[Proposition~3.7]{Din99}. Hence, $\norm{DG} \lesssim \norm{f}$ on $U$ and for all $x, y \in U$ one has by the convexity of $U$
		\begin{equation*}
			\begin{split}
				\MoveEqLeft \normalnorm{u_{\mathcal{P}(x)} - u_{\mathcal{P}(y)}}_{C([0,T];\Tr_p A)} = \biggnorm{\int_0^1 \frac{d}{dz} u_{\mathcal{P}(x + z(y-x))} \d z}_{C([0,T];\Tr_p A)} \\
				& \le \sup_{p \in U} \norm{DG(p)}_{\mathcal{L}(E,C([0,T];\Tr_p A))} \norm{x-y}_E \lesssim \norm{x-y}_E \norm{f}_{L^p([0,T];X)}. \qedhere
			\end{split}
		\end{equation*}
	\end{proof}
	
	\begin{proposition}\label{prop:linfity_initial}
		Let $(\mathcal{P}, F_1 , F_2, E, U, O)$ be an $L^p$-trace parameterization in a UMD space $X$. If $u_{\mathcal{A}}$ solves $\dot{u} + \mathcal{A}u = 0$ for $u(0) = u_0$, then uniformly in $T \in (0, \infty)$
		\begin{equation*}
			\normalnorm{u_{\mathcal{P}(x)} - u_{\mathcal{P}(y)}}_{C([0,T];\Tr_p A)} \lesssim \norm{x-y}_E \norm{u_0}_{\Tr_p A} \qquad \text{for all } x,y \in U.
		\end{equation*}
	\end{proposition}
	\begin{proof}
		Similiarly as before, one obtains that the mapping $z \mapsto u_{\mathcal{A} + z \Delta \mathcal{A}}(t) = e^{-t(\mathcal{A} + z \Delta \mathcal{A})} u_0$ into $F_1$ is analytic. Hence, we again need a uniform estimate on $\norm{u_{\mathcal{B}}(t)}_{\Tr_p A}$ for our approach via complex analysis. We obtain this estimate by using real interpolation for the mapping $u_0 \mapsto [t \mapsto e^{-tB} u_0]$, which is uniformly bounded as mappings $D(B) \to L^{\infty}([0,T]; D(B))$ and $X \to L^{\infty}([0,T];X)$.
	\end{proof}

\section{A Priori Estimates for the Derivatives}

	In this subsection we provide the results for estimate~\eqref{eq:difference_estimate}.		
	
	\begin{proposition}\label{prop:l2}
		Let $(\mathcal{P}, F_1 , F_2, E, U, O)$ be an $L^p$-trace parameterization in a UMD-space $X$. If $u_{\mathcal{A}}$ solves $\dot{u} + \mathcal{A}u = f$ for $u(0) = 0$, then uniformly in $T \in (0,\infty]$
			\begin{equation*}
				\normalnorm{\dot{u}_{\mathcal{P}(x)} - \dot{u}_{\mathcal{P}(y)}}_{L^p([0,T);X)} \lesssim \norm{x-y}_E (\norm{f}_{L^p([0,T);X)} + \norm{u_0}_{\Tr_p A}) \quad \text{for all } x,y \in U.
			\end{equation*}
	\end{proposition}
	\begin{proof}
		Again, this follows from analyticity and the uniform maximal regularity estimate discussed in Remark~\ref{rem:uniform_estimate}.
	\end{proof}
	
\section{Remarks \& Complements}

	We finish with some remarks on the used methods and obtained results. 
	
	\subsection{The \texorpdfstring{$L^p$}{p}-range for maximal regularity}
	
	Most known $L^p$-maximal regularity results are insensitive to the exponent $p$, i.e.\ they hold for one $p \in (1, \infty)$ if and only if they hold for all $p \in (1, \infty)$. In the autonomous case this follows from the theory of singular integrals and is a central fact, whereas in the non-autonomous case one often makes use of Acquistapace--Terreni type results as in~\cite{HaaOuh15} or~\cite{AreMon14}.
	
	In contrast, our results on maximal regularity do not extend to all $p \in (1, \infty)$ even if one deals with symmetric forms. Let us give an easy example based on one dimensional differential operators.
	
	\begin{example}\label{ex:robin}
		For $\beta \in \IC$ consider the forms $a_{\beta}\colon H^1([0,1]) \times H^1([0,1]) \to \IC$ with
    	\begin{equation*}
    		a_{\beta}(u,v) = \int_0^1 u \overline{v} + \int_{0}^1 u' \overline{v'} + \beta u(0) \overline{v(0)}.
    	\end{equation*}
    	The operator $A_{\beta}$ associated to $a_{\beta}$ is $A_{\beta}u = u - u''$ with explicitly known domain $D(A_{\beta}) = \{ u \in H^2([0,1]): u'(0) = \beta u(0) \text{ and } u'(1) = 0 \}$. By~\cite[p.~321, Theorem]{Tri78} we have for $p > 4$
    	\begin{equation*}
    		\begin{split}
    			\Tr_p A_{\beta} & = (L^2([0,1]),D(A))_{1-1/p,p} \\
    			& = \{ u \in B^{2(1-1/p)}_{2p}([0,1]): u'(0) = \beta u(0) \text{ and } u'(1) = 0 \},
    		\end{split} 
    	\end{equation*}
    	whereas $\Tr_p A_{\beta} = B_{2p}^{2(1-1/p)}([0,1])$ for $p < 4$. Hence, for $p > 4$ there is a smooth function $x\colon [0,1] \to \IR$ with $x \in \Tr_p A_0$, but $x \not\in \Tr_p A_1$. Choose $w \in \MaxReg_p^{A_0}([0,1])$ with $w(1) = x$. 
    	Now, let $a(t,u,v) = a_0(u,v) \mathds{1}_{[0,1)} + a_1(u,v) \mathds{1}_{[1,2)}$ and $f = (\dot{w} + A_0 w) \mathds{1}_{[0,1]}$. Then $a$ fails to have maximal $L^p$-regularity: if this would be the case and $u$ is the solution of~\eqref{eq:nacp} for $u(0) = w(0)$, then $u(\cdot - 1)$ solves $\dot{v} + A_1 v = 0$ with $v(0) = w(1) = x$. This implies $x \in \Tr_p A_1$, which is not the case. However, for $p < 4$ the mapping $\IC \to \mathcal{L}(V,V')$ induced by $\beta \mapsto a_{\beta}$ gives rise to an $L^p$-trace parameterization and one can apply~Theorem~\ref{thm:existence} and Theorem~\ref{thm:uniqueness} to obtain maximal $L^p$-regularity. 
	\end{example}
	
	Note that the above argument is generic. It works as long as $\Tr_p A_1 \not\subseteq \Tr_p A_0$. In particular, one sees that the exponent $p = 2$ obtained in Corollary~\ref{cor:elliptic_operator} is optimal for general elliptic operators.
	
	\begin{example}
		Let $a \in L^{\infty}(\IR)$ be \emph{real} with $a \ge \epsilon$ almost everywhere. We consider the elliptic operator induced by the form $a\colon H^1(\IR) \times H^1(\IR) \to \IC$ with
		\begin{equation*}
			a(u,v) = \int_{\IR} a u' \overline{v'} + \int_{\IR} u \overline{v}.
		\end{equation*}
		Since $a$ is symmetric, we clearly have $\Tr_2 A = D(A^{1/2}) = H^{1}(\IR)$. Further, $D(A) = \{ u \in H^1(\IR): au' \in H^{1}(\IR) \}$. It follows from the reiteration theorem for real interpolation that $\Tr_p A = \{ u \in H^{1}(\IR): au' \in B^{1-2/p}_{2p}(\IR) \}$ for $p > 2$. Note that if $a$ is smooth, then $\Tr_p A = \{ u \in H^1(\IR) = B^{1}_{22}(\IR): u' \in B^{1-2/p}_{2p}(\IR) \} = B^{2-2/p}_{2p}(\IR)$. Now, choose an elliptic parameter $a$ with $a \not\in B_{2p}^{1-2/p}([-1,1])$ for all $p > 2$. Then for $u \in C_c^{\infty}(\IR)$ with $u'(x) \in [1,2]$ for all $x \in [-2,2]$ one has $u \not\in \Tr_p(A)$ for $p > 2$, whereas $u$ is clearly in the domain of the second derivative. From the generic argument detailed in Example~\ref{ex:robin} one obtains a non-autonomous form of elliptic operators that fails maximal $L^p$-regularity for all $p > 2$.
	\end{example}
	
	\subsection{Symmetric forms without parameterization}
	
	It is natural to ask whether every family of forms satisfying the Kato square root property uniformly can be parameterized by an $L^2$-trace parameterization. The next example shows that this can be troublesome.
	
	\begin{example}
		In~\cite[Section~4]{McI90} McIntosh constructs symmetric forms $a_t\colon V \times V \to \IC$ for $t \in (-1,1)$ for which $t \mapsto A_t^{1/2} \in \mathcal{L}(V,H)$ is not real-analytic. Now, assume that the forms $a_t$ for $t \in (-1,1)$ are parameterized by an $L^2$-trace parameterization $(\mathcal{P}, F_1 , F_2, E, U, O)$ such that $A_t = \mathcal{P}(x+ty)_{|H}$ for some $x,y \in U$. It then follows along the lines of the previous sections that the map $O \to \mathcal{L}(V,H)$ given by $x \mapsto \mathcal{P}(x)_{|H}^{1/2}$ is analytic. A fortiori, $t \mapsto \mathcal{P}(x+ty)_{|H}^{1/2} = A_t^{1/2}$ is real-analytic contradicting our choice. More generally, no real-analytic selection $\gamma\colon (-1,1) \to U$ of $a_t$ can exist by the same argument. 
	\end{example}
						
	\emergencystretch=0.75em
	\printbibliography

\end{document}